\documentclass[a4paper, reqno, 12pt]{article}

\usepackage[top=25mm, bottom=25mm, left=20mm, right=20mm, a4paper]{geometry}
\usepackage{epsf}
\usepackage{amsmath}
\usepackage{amssymb}
\usepackage{amsfonts}
\usepackage{oldlfont}
\usepackage{graphics}
\usepackage{oldgerm}
\usepackage{latexsym}
\usepackage{dsfont}
\usepackage{amsthm}
\usepackage{graphicx}
\usepackage[affil-it]{authblk}
\usepackage[colorlinks,breaklinks,backref]{hyperref}
\usepackage{epstopdf}
\usepackage{verbatim}
\usepackage{enumitem}
\usepackage{color}

\setenumerate{label={\normalfont(\roman*)}, ref={\normalfont(\roman*)}, itemsep=0pt}

\newtheorem{theorem}{\bf Theorem}[section]
\newtheorem{corollary}[theorem]{\bf Corollary}
\newtheorem{lemma}[theorem]{\bf Lemma}
\newtheorem{proposition}[theorem]{\bf Proposition}

\newtheorem{example}[theorem]{\bf Example}
\newtheorem{question}{\bf Question}

\newtheorem*{IMConjectureForGroups}{\bf The Infinite Motion Conjecture for Permutation Groups}
\newtheorem*{IMConjectureForGraphs}{\bf The Infinite Motion Conjecture for Graphs}
\newtheorem*{MotionConjectureForGaphs}{\bf The Motion Conjecture for Graphs}

\newenvironment{edit}{\color{red}}{\color{black}}
\newcommand{\aed}{\begin{edit}}
\newcommand{\zed}{\end{edit}}
\newenvironment{change}{\color{blue}}{\color{black}}
\newcommand{\ach}{\begin{change}}
\newcommand{\zch}{\end{change}}

\newcommand{\A}{\ensuremath{\operatorname{Aut}}}
\newcommand{\Sym}{\ensuremath{\operatorname{Sym}}}

\newcommand{\D}{\ensuremath{\operatorname{D}}}

\newcommand{\m}{\ensuremath{\operatorname{m}}}

\newcommand{\supp}{\ensuremath{\operatorname{supp}}}
\newcommand{\cl}{\ensuremath{\operatorname{cl}}}

\newcommand{\N}{{\mathbb{N}}}
\newcommand{\Z}{{\mathbb{Z}}}
\newcommand{\Q}{{\mathbb{Q}}}

\title{Infinite Motion and 2-Distinguishability of Graphs and Groups}

\author{Wilfried Imrich}
\affil{Montanuniversit\"at Leoben,

8700 Leoben, Austria.

imrich@unileoben.ac.at}

\author{Simon M.~Smith}

\affil{Mathematics Department,

New York City College of Technology,

City University of New York, USA.

sismith@citytech.cuny.edu}

\author{Thomas W.~Tucker}
\affil{Department of Mathematics, Colgate University,

Hamilton, New York, USA.

ttucker@colgate.edu}

\author{Mark E.~Watkins}
\affil{Department of Mathematics, Syracuse University,

Syracuse, New York, USA.

mewatkin@syr.edu}

\begin{document}
\maketitle

\begin{abstract}
A group $A$ acting faithfully on a set $X$ is $2$-{\em distinguishable} if there is a $2$-coloring of $X$ that is not preserved by any nonidentity element of $A$, equivalently, if there is a proper subset of $X$ with trivial setwise stabilizer. The {\em motion} of an element $a \in A$ is the number of points of $X$ that are moved by $a$, and the {\em motion} of the group $A$ is the minimal motion of its nonidentity elements. For finite $A$, the Motion Lemma says that if the motion of $A$ is large enough (specifically at least $2\log_2 |A|$),
then the action  is $2$-distinguishable.  For many situations where $X$ has a combinatorial or algebraic structure, the Motion Lemma implies the action of $\A(X)$ on $X$ is 2-distinguishable in all but finitely many instances.  

We prove an infinitary version of the Motion Lemma for countably infinite permutation groups, which states that infinite motion is large enough to guarantee $2$-distinguishability. From this we deduce a number of results, including the fact that every locally finite, connected graph whose automorphism group is countably infinite is $2$-distinguishable. One cannot extend the Motion Lemma to uncountable permutation groups, but nonetheless we prove that $2$-distinguishable permutation groups with infinite motion are dense in the class of groups with infinite motion.
We conjecture an extension of the Motion Lemma which we expect holds for a restricted class of uncountable permutation groups, and we conclude with a list of open questions. The consequences of our results are drawn for orbit equivalence of infinite permutation groups.
\end{abstract}



%
%
\section{Introduction}
\label{sec:intro}

The {\em distinguishing number} $\D(A,X)$ of a group $A$ acting faithfully on a set $X$ is the least
cardinal number $d$ such that $X$ has a labeling with $d$
labels that is preserved only by the identity of $A$.
Distinguishability was introduced in \cite{alco-96}; subsequently a large number of papers have been written on this subject (see for example \cite{wazh-07, imkltr-07, LNS, STW} and
compare \cite{BaCa} for related results).  We are mostly interested in cases where $D(A,X)\leq 2$, which means that some proper subset of $X$ has trivial setwise stabilizer in $A$ (or that there is a regular orbit in the action of $A$ on the subsets of $X$).

Distinguishability shares much with the notion of orbit-equivalence. Two permutation groups $A, B \leq \Sym(X)$ are {\em orbit-equivalent} if they have the same orbits on the set of finite subsets of $X$ (see \cite{BaCa} for example); $A$ and $B$ are {\em strongly orbit-equivalent} if they have the same orbits on the entire power set $2^X$.  Note that if $A$ is $2$-distinguishable, then it cannot be strongly orbit-equivalent to any proper subgroup of $A$ (see the proof of Corollary~\ref{cor:evansoecorr}); in this way a $2$-distinguishable permutation group is minimal with respect to strong orbit-equivalence. \\

In many contexts, when the underlying set $X$ is finite, all but finitely many group actions have distinguishing number $2$. This is true for automorphism groups of finite maps (\cite{tu-06}), for finite vector spaces (\cite{klwozh-06}), for groups (\cite{CoTu}), for iterated Cartesian products of a graph (\cite{Al-2005,KlZh,ImKl}), and for transitive permutation groups  having a base (a set whose pointwise stabilizer is trivial) of bounded size (\cite{CoTu, CNS}).

Underpinning these results is the Motion Lemma of Russell and Sundaram (\cite{RuSu}). The {\em motion} of a permutation $a \in A$ (also called its {\em degree}) is the number of elements in $X$ that are not fixed by $a$, and the {\em motion} of $A$ is the infimum of the motion of its nonidentity elements (often called the {\em minimal degree} of $A$).  The Motion Lemma states that if the motion of $A$ is large enough relative to the size of $A$, then $\D(A,X)=2$.

In this paper we explore the extent to which infinite motion affects the distinguishing number of infinite permutation groups. Throughout, $A$ denotes a group of permutations of a countably infinite set $X$.
We prove an infinitary version of the Motion Lemma for countable permutation groups (Lemma~\ref{lem:CountableMotionLemma}), and use it to extend many of the results in \cite{CoTu}. We also prove that a countably infinite permutation group $A \leq \Sym (X)$ that is closed in the permutation topology (see Section~\ref{sec:preliminaries}) has distinguishing number $2$ whenever all orbits of point stabilizers are finite (Theorem~\ref{thm:evanscorr}). From this we obtain the surprising result that any connected, locally finite graph whose automorphism group is countably infinite has distinguishing number $2$ (Corollary~\ref{countable_aut_gp}).
Theorem~\ref{thm:evanscorr} has an interesting corollary for orbit-equivalence: suppose $A$ is an infinite, subdegree-finite, closed permutation group; if $|A| < 2^{\aleph_0}$, then $A$ is not strongly orbit-equivalent to any proper subgroup (Corollary~\ref{cor:evansoecorr}).

Extending the Motion Lemma to uncountable permutation groups fails,
as there exist uncountable permutation groups with infinite motion and infinite distinguishing number (see Section~\ref{sec:uncountable}). Instead, we prove a density result, which states that every group $A \leq \Sym(X)$ with infinite motion is the closure (in the permutation topology) of a $2$-distinguishable group with infinite motion (Theorem~\ref{dense}). Again this has a natural consequence for orbit-equivalence: groups that are minimal with respect to strong orbit-equivalence are dense in the class of permutation groups with infinite minimal degree (Corollary~\ref{cor:oedense}).

We conjecture that any closed group $A \leq \Sym(X)$ with infinite motion is $2$-distinguishable whenever all orbits of its point stabilizers are finite (see {\em The Infinite Motion Conjecture for Permutation Groups} in Section~\ref{sec:uncountable}); this is an extension of a conjecture first made in \cite{tu-06}. In the process of trying to prove this conjecture a number of interesting questions have arisen, which we detail in Section~\ref{sec:questions}.

%
%
\section{Preliminaries}
\label{sec:preliminaries}

We write $(A, X)$ to denote a group $A$ acting on a set $X$. Since all actions in this paper are faithful, we consider $A$ to be a subgroup of the symmetric group $\Sym(X)$. The group $A$ is {\em subdegree-finite} if for all $x\in X$, all orbits of the point-stabilizer $A_x$ in its action on $X$ are finite. Notice that if a graph is locally finite, then its automorphism group must be subdegree-finite. We follow the convention that a set $X$ is {\em countable} if $|X|\leq\aleph_0$ and say that $X$ is {\em countably infinite} if $X$ is both countable and infinite.

For a given coloring of $X$ and a faithful action of $A$ on $X$, we say that $a\in A$ {\em preserves} the coloring if $ax$ and $x$ have the same color for all $x\in X$. The {\em support} of an element $a\in A$, denoted $\supp(a)$, is the set $\{x \in X : ax \not = x\}$; we call $\m(a):=|\supp(a)|$ the {\em motion} of $a$. The {\em motion of the group} $A$ is $\m(A):=\min\{\m(a): a\in A\setminus\{1\}\}$. A group is said to have {\em infinite motion} if $\m(A)$ is any infinite cardinal. A {\em base} of $(A, X)$ is a subset $Y$ of $X$ whose pointwise stabilizer $A_{(Y)}:= \bigcap\{A_y:y\in Y\}$ is trivial. If $G$ is a graph, then $m(G):=m(\A(G))$, where $\A(G)$ is acting on the vertex set $VG$ of $G$.

The following lemma is a useful tool for analyzing the action of infinite, subdegree-finite permutation groups. Our proof assumes the Axiom of Countable Choice.

\begin{lemma}\label{subdegree} Let $A\leq\Sym(X)$, and suppose that $A$ is subdegree-finite and has an infinite orbit. For any finite subsets $Y, Z$ of $X$, there exists an element $a\in A$ such that $Y\cap aZ=\emptyset$
 \end{lemma}
 \begin{proof}  Suppose $x \in X$ lies in an infinite orbit of $A$.  Then there exists an infinite sequence $S=\{a_i \in A : i \in\N \} \subseteq A$ such that $a_i x = a_j x$ if and only if $i = j$.  Suppose there exist finite subsets $Y$ and $Z$ of $X$ such that $Y \cap a_i Z\neq\emptyset$ for infinitely many $i\in\N$.  Then there exist $y \in Y$ and $z \in Z$ and an infinite subsequence $\{a_{i_j} : j \in \N\} \subseteq S$ such that $a_{i_j}z = y$. For each $j\in\N$ define $b_j := a_{i_1}a_{i_j}^{-1}$, and notice that $b_j$ belongs to the stabilizer $A_y$. However, $b_j x = b_k x$  if and only if $j=k$, and so $A_y$ has an infinite orbit, a contradiction.
\end{proof}

Given a subgroup $A \leq \Sym(X)$, it is not hard to imagine how a sequence of permutations $\{a_i : i \in \N\}$ which all lie in $A$ might converge to some other permutation $a$; but this limit may not itself lie in $A$. When studying infinite permutation groups, it is usually necessary to consider separately groups that contain all their limit points and those that do not, as their behavior can be surprisingly different.

Since $X$ is countably infinite, there is a natural topology on $\Sym(X)$ called the {\em permutation topology}. The family of open sets in this topology is generated by the cosets of point-stabilizers of finite subsets of $X$. Thus a sequence of permutations $\{a_i : i \in \N\} \subseteq \Sym(X)$ converges to a permutation $a \in \Sym(X)$ if and only if for every finite set $Y$ there exists $n \in \N$ such that for all $i \geq n$, the permutation $a_i a^{-1}$ fixes $Y$ pointwise.

A subgroup $A \leq \Sym(X)$ is {\em closed} if it is closed in this topology, and for $B \leq A$ the {\em closure of $B$ in} $A$, denoted $\cl_A(B)$, consists of those elements $a\in A$ such that, for any finite subset $Y\subset X$, there exists some $b\in B$ such that $ay=by$ for all $y\in Y$. It is straightforward to show that $\cl_A(B)$ is itself a group. Thus $B \leq \cl_A(B) \leq A$. If $A=\Sym(X)$, we denote the closure of $B$ simply by $\cl(B)$. Notice that a group and its closure have the same orbits on the finite subsets of $X$.

If $G$ is a graph with countable vertex set $X$, then it is easily seen that $\A(G)$ is a closed subgroup of $\Sym(X)$ (see the proof of Corollary~\ref{countable_aut_gp_subdegree_inf}). This is a particular case of the following well-known result: a permutation group $A \leq \Sym(X)$ is closed if and only if it is the (full) automorphism group of a relational structure on $X$ (see \cite{Cameron07}). 

%
%
\section{Motion and countable groups}
\label{sec:motioncountable}

In this section we consider how motion affects the distinguishability of countable permutation groups. 

The distinguishing number of a permutation group $(A, X)$ can be easily influenced by local phenomena.  For example, identifying a vertex in a copy of $K_n$ with just one vertex of any $m$-distinguishable graph ($n>m+1$) will increase its distinguishing number from $m$ to at least $n-1$.  One consequence of large motion is to suppress such local effects.
In \cite{RuSu}, Russell and Sundaram exploit this property, giving a probabilistic proof of their influential Motion Lemma for finite group actions.  We now present an alternative combinatorial argument.

\begin{lemma} [Motion Lemma \cite{RuSu}\,]\label{MotionLemma}
Given a group $A$ acting faithfully on the finite set $X$, if $\m(A)\geq2\log_2|A|$, then $\D(A,X)=2$.
\end{lemma}
\begin{proof} Let $n=|X|$ and let $\m=\m(A)$. For any nonidentity permutation $a\in A$, the number of 2-colorings preserved by $a$ is $2^c$, where $c$ is the number of cycles of $a$ in $X$.  If $a$ preserves some coloring, then all elements of $X$ in the same cycle must be assigned the same color.  There are $n-\m(a)$ singleton cycles, and $\supp(a)$ decomposes
into at most $\m(a)/2$ other cycles, giving
$$c\leq n-\m(a)+\m(a)/2 \leq n-m/2.$$
Thus the total number of colorings preserved by all the various nonidentity elements of $A$ is at most $(|A|-1)2^{n-m/2}$.  If $2^{m/2}>|A|-1$, then the total number of 2-colorings so preserved is less than the total number $2^n$ of 2-colorings.  We've shown that if $2^m\geq |A|^2$, then $\D(A,X)=2$.
\end{proof}

Since any given 2-coloring is likely to be preserved by more than one nonidentity element of $A$, there is a great deal of over-counting in this proof; it is possible to obtain sharper results when additional conditions are placed on $(A, X)$. However, in many contexts the above lemma suffices.

Recall that a  {\em Frobenius group} is a non-regular transitive permutation group in which only the identity fixes two points (and hence every 2-subset is a base).

\begin{proposition}[\cite{CoTu}] \label{list}
 In each of the following instances, all but finitely many of the group actions $(A,X)$ have distinguishing number 2:
 \begin{enumerate}
 \item   \label{list:group}  $X$ is a finite group and $A=\A(X)$;
 \item \label{list:vector}  $X$ is a finite vector space and $A=\A(X)$;
 \item \label{list:base}  $A$ is a finite transitive permutation group having a base of given size;
 \item \label{list:map}   $X$ is the vertex set of a finite map $M$ and $A=\A(M)$;
 \item \label{list:cyclic} $A$ is a finite transitive permutation group with cyclic point-stabilizer;
 \item \label{list:frob} $A$ is a finite Frobenius group.
\end{enumerate}
\end{proposition}

Each proof in \cite{CoTu} is a direct application of the Motion Lemma.
However, one must not presume that elementary applications of the Motion Lemma are always sufficiently powerful to determine the distinguishing number of any class of finite groups. For example, it is known that all but finitely many finite primitive permutation groups that are neither symmetric nor alternating have distinguishing number 2 (see \cite{CNS}).  The proof of this result depends upon the Classification of the Finite Simple Groups. We know of no obvious way to utilize motion to obtain a proof that avoids the Classification.\\

The Motion Lemma tells us that if a group $A$ acts with sufficiently large motion on a finite set $X$, then $(A, X)$ must be 2-distinguishable. If one now considers the situation in which $X$ is infinite, one might guess that infinite motion is ``sufficiently large" to ensure $D(A, X) = 2$. For countably infinite groups, this guess is correct (Lemma~\ref{lem:CountableMotionLemma}), but as will be seen in Section~\ref{sec:uncountable}, this is not true for groups with larger cardinality.

\begin{lemma} \label{lem:CountableMotionLemma}
Let $A\leq\Sym(X)$, where both $A$ and $X$ are countably infinite.  If $A$ has infinite motion, then $\D(A,X)=2$.
\end{lemma}

\begin{proof}
Enumerate the set $X$ as $\{x_i: i \in \N\}$, and the elements of $A$ as $\{a_i: i\in\N\cup\{0\}\}$, with $a_0$ being the identity. Let $i$ be the least natural number such that $x_i \in \supp(a_1)$, and define $y_1 := x_i$. Proceeding inductively, suppose we know $y_1, \ldots, y_n$ and define $y_{n+1}:=x_i$ where $i$ is the least natural number such that $x_i \in \supp(a_{n+1}) \setminus \{y_1, a_1 y_1, \ldots, a_{n-1}y_{n-1}, a_n y_n\}$.  Such a natural number $i$ always exists because $\supp(a_{n+1})$ is infinite. The sets $Y:=\{y_n : 0 < n < |A|\}$ and $Y':=\{a_n y_n:0 < n < |A|\}$ are disjoint subsets of $X$, and every nonidentity element in $A$ moves an element in $Y$ to an element in $Y'$. Thus, coloring the elements of $Y$ black and the elements of $X \setminus Y$ white describes a distinguishing coloring of $(A, X)$.
\end{proof}

Using this lemma, we obtain the following infinitary version of Proposition~\ref{list}.

\begin{theorem} \label{thm:infinitelist}
 In each of the following instances, the group action $(A,X)$ is $2$ distinguishable:
\begin{enumerate}
 \item   \label{list:infgroup}  $X$ is a finitely generated, infinite group and $A=\A(X)$;
 \item \label{list:infvector} {\normalfont (\cite[Theorem 3.1]{Ch2})} $X$ is a finite-dimensional vector space over an infinite field $\mathfrak{K}$, and $A=\A(X)$ (i.e. the general linear group on $X$);
 \item \label{list:infbase}   $X$ is countable and $A$ is an infinite, subdegree-finite permutation group having a finite base;
\item \label{list:infmap}   $X$ is the vertex set of a locally finite, connected map $M$ with at least one vertex of valence at least 3, and $A=\A(M)$ is infinite;
 \item \label{list:infcyclic} $X$ is countable and $A$ is an infinite, subdegree-finite permutation group with a cyclic point-stabilizer;
 \item \label{list:inffrob} $X$ is countable and $A$ is an infinite subdegree-finite Frobenius group.
  \end{enumerate}
\end{theorem}

\begin{proof} \ref{list:infgroup} Let $S$ be any finite generating set for $X$.  Then $X$ is countably infinite because only finitely many words of any given finite length can be formed from $S$.  Since an automorphism of $X$ is determined by its action on $S$, the group $A$ is countable. Each non-identity element of $A$ has infinite motion, because the set of elements of $X$ fixed by any nonidentity automorphism is a proper subgroup of $X$ and hence has an infinite complement. If $A$ is finite but non-trivial, then $D(A, X) = 2$ by Lemma~\ref{MotionLemma}, while if $A$ is countably infinite, then $D(A, X) = 2$ by Lemma~\ref{lem:CountableMotionLemma}.

\ref{list:infvector} Since the set of fixed points of any nonidentity element of $A$ forms a proper subspace of $X$, we infer that $\m(A)$ is infinite. If $\frak{K}$ is countably infinite, then so is $A$, and the result follows by Lemma~\ref{lem:CountableMotionLemma}. However, we present an  argument that holds for a field of arbitrary cardinality.

Let $\{u_1,\ldots, u_n\}$ be an ordered basis for $X$.  There exists an element $c\in \frak{K}$ whose multiplicative order is greater than $n^2$, because $\frak{K}$ is infinite while there are only finitely many solutions $x$ to the equation $x^m=1$ for $0\leq m\leq n$.  Let $Y=\{ c^iu_j : 0\leq i <j \leq n\}$.  Since the 1-dimensional subspace $\langle u_j\rangle$ contains exactly $j$ elements of $Y$,  any linear transformation $a$ that stabilizes $Y$ setwise must, for each $j$, fix $\langle u_j \rangle$ and permute the vectors in $\{c^iu_j : 0 \leq i < j\}$. Since $c$ has order greater than $n$, the only possibility is that $a$ fixes each basis element $u_i$.  Hence $a$ is the identity transformation.

\ref{list:infbase} Suppose $Y \subset X$ is a finite base for $A$; that is, the pointwise stabilizer $A_{(Y)}$ is trivial. Since $X$ is countable and $Y$ is finite, there are countably many images of $Y$ in the set $AY=\{aY:a\in A\}$.  But $|A| = |A : A_{(Y)}| = |AY|$, and so $A$ is countably infinite.  Because $AY$ is infinite and $Y$ is finite, $A$ has an infinite orbit $Ay$ for some $y\in Y$.  Let $b$ be any nonidentity element of $A$ and suppose that $Z=\supp(b)$ is finite.  Then by Lemma \ref{subdegree}, for some $a\in A$, we have $aZ\cap Y=\emptyset$.  But then $aba^{-1}$ has support disjoint from $Y,$ and hence fixes $Y$ pointwise.  However, $aba^{-1}\neq1$, a contradiction since $Y$ is a base.  Hence $A$ has infinite motion.  Since $A$ is also countably infinite, $\D(A,X)=2$ by Lemma~\ref{lem:CountableMotionLemma}.

\ref{list:infmap} A well known property of maps is that if vertex $x$ has valence at least 3, then $x$ has neighbors $y,z$ such that $\{x,y,z\}$ has trivial pointwise stabilizer (these three vertices define a ``corner" of a face).  Thus $A$ has a finite base, and so $\D(A,X)=2$ by part \ref{list:infbase} above.

\ref{list:infcyclic} If some stabilizer $A_x$ is cyclic and not trivial, then $A_x = \langle a \rangle$ for some $a \in A$ and there exists $y \in X$ such that $a y \not = y$.  Thus $\{x,y\}$ is a base for $A$, and the result follows again from part \ref{list:infbase}.

\ref{list:inffrob} By definition, $A$ is an infinite subdegree-finite permutation group having a base of size $2$, and so $\D(A,X)=2$ by part \ref{list:infbase}.
\end{proof}

Notice that very little about the $2$-distinguishability of $(A, X)$ can be deduced from only the existence or non-existence of a finite base for $(A, X)$. There are examples of $2$-distinguishable groups with no finite base (e.g. $T$ is an infinite, locally finite, homogeneous tree and $X = VT$ with $A = \A T$).  There are also examples of groups with a finite base that are not $2$-distinguishable (e.g., the disjoint union $G$ of a complete graph $K_n$ and a double ray has a base of size $n+3$ and distinguishing number $n$; if one wants a connected example, just add an extra vertex adjacent to all vertices of $G$).\\

The following result of
D.~M.~Evans (\cite{Evans87}) is traceable independently to D.~W.~Kueker (\cite[Theorem 2.1]{kreker}) and G.~E.~Reyes (\cite[Theorem 2.2.2]{Reyes}).  It is independent of the Continuum Hypothesis.

\begin{theorem}[{\cite[Theorem 1.1]{Evans87}}] \label{evans} Suppose $X$ is a  countably infinite set. If $A$ and $B$ are closed subgroups of $\Sym(X)$ and $B \leq A$, then either $|A:B| = 2^{\aleph_0}$ or $B$ contains the pointwise stabilizer in $A$ of some finite subset of $X$.  \qed
\end{theorem}

Using this theorem, one can often determine the distinguishing number of countably infinite permutation groups without explicitly requiring them to have infinite motion.

\begin{theorem}\label{thm:evanscorr}  Let $A$ be a closed permutation group on a countably infinite set $X$ such that $\aleph_0\leq|A|<2^{\aleph_0}$.  Then the following two statements hold and are independent of the Continuum Hypothesis.
\begin{enumerate}
\item  \label{list:Evans'Consequence1}  $|A|=\aleph_0$ and $A$ has a finite base and finite distinguishing number.
\item \label{list:Evans'Consequence2} If $A$ is subdegree finite, then $(A,X)$ has infinite motion and $D(A,X)=2$.
\end{enumerate}
\end{theorem}
\begin{proof}
\ref{list:Evans'Consequence1}  Let $B:=\langle1\rangle$, which is obviously closed in $A$.  Then Theorem~\ref{evans} implies that $A$ has a finite base $Y$.  Thus $D(A,X)\leq|Y|<\aleph_0$.

\ref{list:Evans'Consequence2}   Suppose that $A$ is subdegree finite.  By part \ref{list:Evans'Consequence1}, $A$ is countable and has a finite base, and so we may apply Theorem~\ref{thm:infinitelist}~\ref{list:infbase} to deduce that $D(A,X)=2$. That $A$ has infinite motion follows from the proof of Theorem~\ref{thm:infinitelist}~\ref{list:infbase}.
\end{proof}

We deduce from this theorem two corollaries describing the distinguishing number of graphs whose automorphism group is countable.

\begin{corollary}\label{countable_aut_gp_subdegree_inf}  Suppose $G$ is a graph whose vertex set is countably infinite. If  $\aleph_0 \leq |\A(G)| < 2^{\aleph_0}$, then $\A(G)$ is closed and has a finite
base, $|\A(G)| = \aleph_0$ holds,  and  $\D(G)$ is finite.  This statement is independent of the Continuum Hypothesis.
\end{corollary}
\begin{proof} As remarked previously, $\A(G)$ is a closed permutation group. To see why, suppose $\{ a_i : i \in \N \} \subseteq \A(G)$ is a sequence that converges to some $a \in \Sym(VG)$. Given any edge $\{x, y\}$ in $G$, for all sufficiently large $i$ the image of this $2$-element subset of $VG$ under both $a$ and $a_i$ is the same; in particular the image of this edge under $a$ is an edge in $G$, and so $a \in \A(G)$. The corollary follows immediately from Theorem~\ref{thm:evanscorr}~\ref{list:Evans'Consequence1}.
\end{proof}

\begin{corollary}\label{countable_aut_gp}  Let $G$ be a locally finite, connected graph such that  $\aleph_0 \leq |\A(G)| < 2^{\aleph_0}$.  Then $\A(G)$ has a finite base, $|\A(G)| = \aleph_0$,  $\m(G) = \aleph_0$, and  $\D(G)=2$. This holds independently of the Continuum Hypothesis.
\end{corollary}
\begin{proof} Suppose $G$ is a connected, locally finite graph, and $x \in VG$. Each sphere of radius $n \in \N$ centered at $x$ is finite and every vertex lies in one of these $n$-spheres. Each orbit of the stabilizer $\A(G)_x$ is a subset of some $n$-sphere centered at $x$, and is therefore finite. Thus $\A(G)$ is subdegree-finite and closed, and the corollary follows from Theorem~\ref{thm:evanscorr}.
\end{proof}

Corollary~\ref{countable_aut_gp} also follows from the elegant result of R.\,Halin (\cite[Theorem 6]{Hal-73}) that the automorphism group of a locally finite connected graph is uncountable if and only if it has no finite base.

For orbit-equivalence, Theorem~\ref{thm:evanscorr} implies the following.

\begin{corollary} \label{cor:evansoecorr} Suppose that $A$ is an infinite subdegree-finite, closed permutation group.  If  $A$ is strongly orbit-equivalent to some proper subgroup, then $A$ is uncountable.
\end{corollary}
\begin{proof} Suppose that $\aleph_0\leq |A|<2^{\aleph_0}$ and that $B < A$ with $A$ and $B$  strongly orbit-equivalent. Because $A$ is $2$-distinguishable (by Theorem~\ref{thm:evanscorr}), there exists a set $Y \subseteq X$  whose setwise stabilizer $A_{\{Y\}}$ is trivial.  If $a \in A$, then there exists $b \in B$ such that $aY = bY$, because $A$ and $B$ are strongly orbit-equivalent. Hence $a^{-1}b \in A_{\{Y\}}$ and so $a = b$.  Thus $A\subseteq B$, a contradiction.
\end{proof}

Notice that subdegree-finite permutation groups are locally compact subgroups of $\Sym(X)$. The authors of \cite{ML} conjectured that if $X$ is countably infinite and $A, B \leq \Sym(X)$ are strongly orbit-equivalent and closed, then $A = B$. They comment that even knowing whether this is true for locally compact groups would be interesting. The above corollary establishes this conjecture for subdegree-finite countable groups when $B$ is a subgroup of $A$.\\

 \begin{corollary}\label{fg} Let $A$ be an infinite, closed, subdegree-finite permutation group on a countably infinite set $X$  If all point-stabilizers in $A$ are countable, then they are all finite, and $D(A, X) = 2$. 
\end{corollary}
 \begin{proof}  Since $A$ is infinite and closed, $X$ is countably infinite, and all point-stabilizers in $A$ are countable, it folows that $A$ is countably infinite, in which case Theorem \ref{thm:evanscorr}~\ref{list:Evans'Consequence1} applies.  That is, $A$ has a finite base $Y$.  If additionally $A$ is subdegree-finite, then $D(A,X)=2$ by Theorem \ref{thm:evanscorr}~\ref{list:Evans'Consequence2}.  Finally, $A_x$ is finite for all $x\in X$, since orbits of elements of $Y$ under any point-stabilizer $A_x$ are finite.
\end{proof}

It follows that $D(A,X)=2$ holds in the particular case where the point-stabilizers are finitely generated

%
%
\section{Motion and uncountable groups}
\label{sec:uncountable}

Lemma~\ref{lem:CountableMotionLemma} and its consequences are extensions to countable groups of Russell and Sundaram's Motion Lemma. One might be tempted to try to extend this lemma to groups of higher cardinality: {\em If $A$ is a group acting faithfully with infinite motion on the countably infinite set $X$, then $D(A, X) = 2$.}  However, this statement is false: the group $\A(\Q, <)$, which consists of those permutations of the rational numbers which preserve the usual ordering $<$, has infinite motion on $\Q$ and is known to have distinguishing number $\aleph_0$ (\cite{LNS}).  Note that $\A(\Q, <)$ is the automorphism group of the directed graph on $\Q$, with an edge directed from $x$ to $y$ if and only if  $x<y$; it is therefore easily seen to be closed and not subdegree-finite. The only examples we have which show that Lemma~\ref{lem:CountableMotionLemma} cannot be extended to groups of higher cardinality are not subdegree-finite. For this reason we propose the following conjecture.

\begin{IMConjectureForGroups}
If $A$ is a closed, subdegree-finite permutation group with infinite motion on a countably infinite set $X$, then $\D(A,X)=2$.
\end{IMConjectureForGroups}

Since automorphism groups of connected, locally finite graphs are closed and subdegree-finite, the above conjecture implies the following conjecture, originally posed in  \cite{tu-06}.

\begin{IMConjectureForGraphs}
If $G$ is an infinite, locally finite, connected graph and if $\A(G)$ has infinite motion, then $\D(G)=2$.
\end{IMConjectureForGraphs}

Although the local finiteness of $G$ implies the subdegree-finiteness of $\A(G)$, the converse obviously does not hold. Moreover, $\A(\Q,<)$ gives rise to a directed graph with infinite motion that is not locally finite and is not finitely distinguishable, but we know of no (undirected) graph with these properties.

The Infinite Motion Conjecture for Graphs has been studied in the context of growth.  For example, it is not difficult to show that a graph with linear growth and infinite motion is $2$-distinguishable. The Infinite Motion Conjecture is known to hold for locally finite graphs with super-linear but subquadratic growth \cite{cuimle-xx}.  Further refinements of these methods show that the conjecture holds for locally finite graphs with superpolynomial but subexponential growth \cite{Le-xx}.

The Infinite Motion Conjecture for Graphs may fail for graphs of larger cardinality.  For example, consider the star $S_{\textswab n}$ which, for some cardinal number $\textswab n > \aleph_0$, consists of $\textswab{n}$ rays emanating from a single common vertex. Then $\A(S_{\textswab n})$ has infinite motion, but $\D(S_{\textswab n}) > 2$, because in any 2-coloring of $S_{\textswab n}$ there must be at least two rays with the same 2-coloring. But then the automorphism that interchanges them and fixes all other vertices is nontrivial but preserves the coloring.  \\

When applied to graphs, the Motion Lemma \ref{MotionLemma} asserts that a graph $G$ with motion $\m(G)$ is 2-distinguishable  if  $\m(A)\geq 2\log_2|A|$.  When $\m(G)$ is infinite, this inequality is equivalent to  $2^{\m(G)} \geq |\A(G)|$ and leads to the Motion Conjecture for Graphs of \cite{cuimle-xx}.

\begin{MotionConjectureForGaphs} Let $G$ be an infinite graph of arbitrary cardinality. Then $2^{\m(G)} \geq |\A(G)|$ implies that $\D(G) = 2$.
\end{MotionConjectureForGaphs}

We  now turn to proving that $2$-distinguishable groups are dense in the class of permutation groups with infinite motion. We begin by proving for completeness two well-known facts relating the properties of having a finite base, being subdegree-finite, and closure.  

\begin{proposition}  If $A \leq \Sym(X)$ has a finite base, then $A$ is closed.
\end{proposition}
\begin{proof}  If $Y$ is a finite base for $A$ and $a \in \cl(A)$, then there exists $b \in A$ such that $a$ and $b$ agree on $Y$.  Since $ab^{-1} \in \cl(A)$ and $Y$ is finite, for each $x \in X$ there exists $c \in A$ such that $c$ and $ab^{-1}$ agree on $Y \cup \{x\}$.  But any such $c$ must fix the base $Y$ pointwise and is therefore the identity. Thus $ab^{-1}$ must fix every element of $X$; that is, $a=b \in A$. \end{proof}

\begin{proposition}  If $B \leq A\leq \Sym(X)$ and $B$ is subdegree-finite, then $\cl_A(B)$ is also subdegree-finite.
\end{proposition}
\begin{proof} If $x,y,z\in X$ and if $a\in\cl_A(B)$ satisfies $ax=x$ and $ay=z$, then there exists some element $b\in B$ that also satisfies $bx=x$ and $by=z$; hence the stabilizers $B_x$ and $[\cl_A(B)]_x$ have exactly the same orbits.
\end{proof}

On the other hand, as shown by the following example, neither of the two properties of 2-distinguishability nor having infinite motion is preserved under closure.

 \begin{example}\label{PxK2}{\em  Let the graph $G$ be the strong product $P\boxtimes K_2$, where $P$ is the double ray. We  label $VG$ as the disjoint union $\{x_i : i\in \Z\} \cup \{y_i : i\in \Z\}$, where two distinct vertices in $VG$ are adjacent in $G$ if and only if their indices differ by at most $1$.  Let $A=\A(G)$.  Let $b$ be the translation that adds $1$ to all subscripts, and let be $c$ the reflection that multiplies each subscript by $-1$.   Finally, let $d$ be the involution that transposes $x_i$ and $y_i$ if and only if $i$ is a square.  Let $B=\langle b,c,d \rangle$. We list some properties of $A$ and $B$.
\begin{enumerate}
 \item For each $i$, $A$ contains the transposition $s_i$ whose support is $\{x_i,y_i\}$, and so $\m(A)=2$.  Moreover, $A$ is uncountable since it contains all finite and all infinite products of the transpositions $s_i$.
 \item  $\D(A,X)>2$. Given any 2-coloring of $VG$, if some $x_i$ and $y_i$ have the same color, then $s_i$ preserves the coloring.  Otherwise, for some product $t$ of some of the $s_i$, all the images of the $x_i$ have one color and all the images of the $y_i$ have the other color.  In this case, $tbt$ preserves the given coloring.
 \item  $B$ has infinite motion, and $A$ is subdegree-finite.
 \item $\D(B,X)=2$  by Lemma~\ref{lem:CountableMotionLemma}, since $B$ is  countably infinite with infinite motion.
 \item  $s_0 \in \cl_A(B)$. Let $Y$ be any finite subset of $VG$.   For some $n>0$ we have $x_i, y_i \not\in Y$ whenever $|i|\geq n$.  Then $b^{-n^2}db^{n^2}\in B$ and agrees with $s_0$ on $Y$, since $d$ interchanges $x_i$ with $y_i$ for $i=n^2$ and fixes $x_j$ and $y_j$ for $(n-1)^2<j<(n+1)^2, j\neq n^2$.
\item $\cl_A(B)=A$, since $s_i=b^{i}s_0b^{-i}\in \cl_A(B)$ for all $i$.
 \end{enumerate}}
 \end{example}

Although $2$-distinguishability and infinite motion are not preserved under closure, permutation groups with distinguishing number 2 are, in a certain sense, dense in the class of permutation groups with infinite motion. Our proof assumes the Axiom of Choice.

 \begin{theorem} \label{dense} Let $A$ be a group of permutations of the countably infinite set $X$, and suppose $A$ has infinite motion. Then $A=\cl_A(B)$ for some subgroup $B$ satisfying $\D(B,X)=2$.
 \end{theorem}

 \begin{proof} This theorem follows from the fact that, when $X$ is countably infinite, there is a countable subgroup $B$ of $A$ such that $\cl_A(B)=A$; since $A$ has infinite motion, so does $B$, and therefore $\D(B,X)=2$ by Lemma~\ref{lem:CountableMotionLemma}.

The existence of such a group $B$ is well-known, but we describe its construction here for completeness. Since $X$ is countably infinite, we may enumerate the family of its finite subsets as $\{Y_i : i \in \Z\}$. The set $\{(Y_i, aY_i) : a \in A, i \in \Z\}$ is also countably infinite, and so (assuming the Axiom of Choice) we may choose a countably infinite set of elements $\{a_{i j} : i, j \in \Z\} \subseteq A$ such that $\{(Y_i, a_{i j} Y_i) : i, j \in \Z\} = \{(Y_i, aY_i) : a \in A, i \in \Z\}$. Let $B$ denote the group generated by the set $\{a_{i j} : i, j \in \Z\}$.  Then $B$ and $A$ have the same action on the finite subsets of $X$ and therefore $\cl_A(B) = A$. Since $B$ is countably generated, $|B| \leq \aleph_0$. \end{proof}

Theorem \ref{dense} suggests that it might be possible to obtain a proof of the Infinite Motion Conjecture for Permutation Groups by ``bootstrapping" from a countably infinite subgroup of an arbitrary group $A$ of which $A$ is the closure.  Unfortunately, Example~\ref{PxK2} shows that $2$-distinguishability need not be preserved under closure, even  under the strong assumption of subdegree-finiteness.  On the other hand, in Example~\ref{PxK2}, the larger group $A$ has finite motion.  Thus, a possibility remains that a bootstrapping argument may be effective.\\

The following corollary of Theorem~\ref{dense} shows that groups that are minimal with respect to strong orbit-equivalence are dense in the class of permutation groups with infinite minimal degree.

\begin{corollary} \label{cor:oedense} Let $A$ be a group of permutations of the countably infinite set $X$, and suppose the minimal degree of $A$ is infinite. Then there exists a subgroup $B \leq A$ satisfying
\begin{enumerate}
\item
	$A = \cl_A(B)$; and
\item
	$B$ is not strongly orbit-equivalent to any proper subgroup.
\end{enumerate}
\end{corollary}

\begin{proof} Since $A$ has infinite minimal degree, it has infinite motion. From the proof of Theorem~\ref{dense} there exists a subgroup $B \leq A$ such that $|B| \leq \aleph_0$ and $A = \cl_A(B)$ and $D(B, X) = 2$. Hence, $B$ is not strongly orbit-equivalent to any proper subgroup (see the proof of Corollary~\ref{cor:evansoecorr}).
\end{proof}

%
%
\section{Questions}
\label{sec:questions}

We list four of the various questions that have arisen in the course of preparing this article.
\medskip

As remarked just after Proposition \ref{list}, the only known proof (cf.~\cite{CNS}) that there are only finitely many non-alternating, non-symmetric finite primitive permutation groups that are not $2$-distinguishable relies on the Classification of the Finite Simple Groups, and the link between primitivity and $2$-distinguishably (and thus orbit-equivalence) is not fully understood. Such a connection has also been observed in infinite groups (\cite[Corollary 2]{STW}) and is the subject of two interesting conjectures (\cite[Section 6]{LNS} and \cite[Conjecture 1.2]{ML}). Thus, the answer to the following question is of considerable interest.

\begin{question}  Is there an elementary proof, perhaps based on motion, that there are only finitely many finite primitive permutation groups, other than $A_n$ and $S_n$, satisfying $\D(A,X)>2$?
\end{question}

We do not know if the Infinite Motion Conjecture for Graphs requires the hypothesis that $G$ be locally finite.

\begin{question} \label{Q:LocallyFiniteGraphs}   Does there exist a connected graph $G$ with a countably infinite vertex set, such that $G$ is not locally finite, $G$ has infinite motion, and $\D(G)>2$?
\end{question}

We do not know whether the condition of being closed is necessary in the Infinite Motion Conjecture for Permutation Groups.  We thus expect an affirmative answer to the following question.

\begin{question}  Does there exist a subdegree-finite, non-closed permutation group $A$ on a countably infinite set $X$ such that $A$ has infinite motion and $\D(A,X)>2$?
\end{question}

Some of the finite results listed in Proposition~\ref{list} require transitivity, but their infinite analogues never mention transitivity.  If we require transitivity, can we then drop the requirement of closure or subdegree-finiteness?

\begin{question} Does there exist a transitive permutation group $A$ on a countably infinite set $X$ all of whose point-stabilizers are finitely generated and $\D(A,X)>2$?
\end{question}

Regarding this question, suppose that $A$ is transitive and $A_x$ is finitely generated for all $x \in X$. Since $X$ is countably infinite, it follows that $A$ is countably infinite. Thus if $A$ has infinite motion, then $\D(A,X)=2$ (by Lemma \ref{lem:CountableMotionLemma}).  On the other hand, if $A$ does not have infinite motion, the transitivity of $A$ makes it unlikely that $A_x$ is finitely generated for all $x$. Certainly if $A$ is closed and transitive then it must have infinite motion (for otherwise $A$ would be uncountable).\\

\vspace{3mm}

\noindent {\bf Acknowledgements}\\

Wilfried Imrich was partially supported by the Austrian Science Fund (FWF), project W1230, and ARRS Slovenia within the EUROCORES
Programme EUROGIGA/GReGAS of the European Science Foundation. Mark Watkins was partially supported by a grant from the Simons Foundation (\#209803 to Mark E.~Watkins).

%
%

\end{document}